\documentclass[10pt]{amsart}

\usepackage[latin1]{inputenc}
\usepackage[english]{babel}
\usepackage{indentfirst}
\usepackage{amssymb}
\usepackage{amsthm}
\usepackage[]{graphicx}
\usepackage[all]{xy}

\newcommand{\Ba}{{\rm Baire}}
\newcommand{\impli}{\Rightarrow}
\newcommand{\Nat}{\mathbb{N}}

\newcommand{\cD}{\mathcal{D}}

\def\epsilon{\varepsilon}

\newcommand{\sub}{\subseteq}

\newtheorem{theo}{Theorem}[section]
\newtheorem{lem}[theo]{Lemma}
\newtheorem{pro}[theo]{Proposition}
\newtheorem{cor}[theo]{Corollary}
\newtheorem{fact}[theo]{Fact}

\newtheorem{question}[theo]{Question}

\theoremstyle{definition}
\newtheorem{rem}[theo]{Remark}
\newtheorem{defi}[theo]{Definition}

\numberwithin{equation}{section}

\title{On integration in Banach spaces and total sets}

\author{Jos\'e Rodr\'{i}guez}
\address{Dpto. de Ingenier\'{i}a y Tecnolog\'{i}a de Computadores,
Facultad de Inform\'{a}tica, Universidad de Murcia, 30100 Espinardo (Murcia), Spain}
\email{joserr@um.es}



\thanks{Research supported by projects MTM2014-54182-P, MTM2017-86182-P (AEI/FEDER, UE)
and 19275/PI/14 (Fundaci\'on S\'eneca).}

\dedicatory{Dedicated to the memory of Joe Diestel}

\begin{document}

\begin{abstract}
Let $X$ be a Banach space and $\Gamma \sub X^*$
a total linear subspace. We study the concept of $\Gamma$-integrability for $X$-valued functions~$f$ defined on a complete probability space, 
i.e. an analogue of Pettis integrability by dealing only with the compositions $\langle x^*,f \rangle$ for~$x^*\in \Gamma$.
We show that $\Gamma$-integrability and Pettis integrability are equivalent whenever $X$ has Plichko's property~($\mathcal{D}'$) (meaning that
every $w^*$-sequentially closed subspace of~$X^*$ is $w^*$-closed).
This property is enjoyed by many Banach spaces including all spaces with $w^*$-angelic dual as well as all spaces which 
are $w^*$-sequentially dense in their bidual. 
A particular case of special interest arises when considering $\Gamma=T^*(Y^*)$
for some injective operator $T:X \to Y$. Within this framework, we show that if $T:X \to Y$ is a semi-embedding, $X$ has property~($\mathcal{D}'$)
and $Y$ has the Radon-Nikod\'{y}m property, then $X$ has the weak Radon-Nikod\'{y}m property. This extends
earlier results by Delbaen (for separable $X$) and Diestel and Uhl (for weakly $\mathcal{K}$-analytic~$X$).
\end{abstract}

\maketitle

\section{Introduction}

A result attributed to Delbaen, which first appeared in a paper by Bourgain and Rosenthal (see \cite[Theorem~1]{bou-ros-2}),
states that if $T:X\to Y$ is a semi-embedding between Banach spaces (i.e. an injective operator such that $T(B_X)$ is closed), 
$X$~is separable and $Y$ has the Radon-Nikod\'{y}m property (RNP), 
then $X$ has the RNP as well (cf. \cite[Theorem~4.1.13]{bou-J}). That result was also known 
to be true if $X$ is weakly $\mathcal{K}$-analytic, see \cite[footnote on p.~160]{die-uhl-2}. 
No proof nor authorship info of that generalization was given in \cite{die-uhl-2}. 
In our last email exchange I asked Prof. Joe Diestel about that and he told me:
\begin{quote} 
{\em ``The result was realized as so, around the time that semi-embeddings were being properly appreciated, 
as almost immediate consequences of the more general notions of $\mathcal{K}$-analyticity.  
Jerry and I understood what we understood thru Talagrand's papers. So if attribution is the issue, it's Talagrand's fault!''}.
\end{quote}
Loosely speaking, a key point to get such kind of results is to deduce the integrability of an $X$-valued function~$f$
from the integrability of the $Y$-valued composition $T\circ f$, which quite often reflects on the family
of real-valued functions 
$$
	\{\langle y^*, T \circ f \rangle: \, y^*\in Y^*\}=
	\{\langle x^*, f \rangle: \, x^*\in T^*(Y^*)\},
$$
i.e. the compositions of~$f$ with elements of the total linear subspace $T^*(Y^*) \sub X^*$.

In this paper we study Pettis-type integration of Banach space valued functions
with respect to a total linear subspace of the dual.
Throughout $(\Omega,\Sigma,\nu)$ is a complete probability space and $X$ is a Banach space. Let $\Gamma \sub X^*$
be a total linear subspace and let $f:\Omega\to X$ be a function. Following~\cite{mus6}, $f$ is said to be: (i)~{\em $\Gamma$-scalarly integrable} if $\langle x^*, f \rangle$ is 
integrable for all $x^*\in \Gamma$; (ii) {\em $\Gamma$-integrable} if it is $\Gamma$-scalarly integrable and for every $A\in \Sigma$ there is 
an element $\int_A f \, d\nu \in X$ such that $x^*(\int_A f \, d\nu)=\int_A \langle x^*, f \rangle \, d\nu$ for all $x^*\in \Gamma$.
This generalizes the classical Pettis and Gelfand integrals, which are obtained respectively when $\Gamma=X^*$ or $X$ is a 
dual space and $\Gamma$ is the predual of~$X$. When does $\Gamma$-integrability imply Pettis integrability? 
Several authors addressed this question in the particular case of Gelfand integrability.
For instance, Diestel and Faires (see \cite[Corollary~1.3]{die-fai})
proved that Gelfand and Pettis integrability coincide for any strongly measurable function
$f:\Omega \to X^*$ whenever $X^*$ contains no subspace isomorphic to~$\ell_\infty$,
while Musia{\l} (see \cite[Theorem~4]{mus6}) showed that one can give up strong measurability if 
the assumption on~$X$ is strengthened to being $w^*$-sequentially dense in~$X^{**}$.
More recently, $\Gamma$-integrability has been studied in \cite{and-zie,kun}
in connection with semigroups of operators. 

This paper is organized as follows. 

Section~\ref{section:measurability} contains some preliminaries on scalar measurability of Banach space valued functions and total sets. 
Special attention is paid to Banach spaces having Gulisashvili's {\em property~($\cD$)} \cite{gul-J}, i.e. those
for which scalar measurability can be tested by using any total subset of the dual.

In Section~\ref{section:integration} we analyze $\Gamma$-integrability and
discuss its coincidence with Pettis integrability. We show (see Theorem~\ref{theo:Dprime}) that this is the case whenever $X$ 
has Plichko's {\em property ($\mathcal{D}'$)}~\cite{pli3}, which means that every $w^*$-sequentially closed
subspace of~$X^*$ is $w^*$-closed. By the Banach-Dieudonn\'{e} theorem,
property~($\mathcal{D}'$) is formally weaker than {\em property~($\mathcal{E}'$)} of~\cite{gon3},
which means that every $w^*$-sequentially closed convex bounded subset of~$X^*$ is $w^*$-closed.
The class of Banach spaces having property ($\mathcal{E}'$) includes those with $w^*$-angelic dual
as well as all spaces which are $w^*$-sequentially dense in their bidual
(see \cite[Theorem~5.3]{avi-mar-rod}). On the other hand, we also study conditions ensuring the $\Gamma$-integrability
of a $\Gamma$-scalarly integrable function. This is connected in a natural way to the Mazur property of $(\Gamma,w^*)$ and the completeness of 
the Mackey topology $\mu(X,\Gamma)$ associated to the dual pair $\langle X,\Gamma \rangle$, which have been topics of recent
research in \cite{bon-cas,gui-mon,gui-mon-ziz}. If $X$ has property~($\mathcal{E}'$) and $\Gamma$ is norming, 
we prove that a function $f:\Omega \to X$ is Pettis integrable whenever the family 
$\{\langle x^*, f \rangle: \, x^*\in \Gamma \cap B_{X^*}\}$ is uniformly integrable (Theorem~\ref{theo:ScalarGamma}).

Finally, in Section~\ref{section:operators} we focus on the interplay between the integrability
of a function $f:\Omega \to X$ and that of the composition $T\circ f:\Omega \to Y$, where $T$
is an injective operator from~$X$ to the Banach space~$Y$. We extend
the aforementioned results of Delbaen, Diestel and Uhl by proving that if $T$ is a semi-embedding, $Y$
has the RNP and $X$ has property~($\mathcal{D}'$), then $X$ has the weak Radon-Nikod\'{y}m property (WRNP),
see Theorem~\ref{theo:sequential}. In this statement, the RNP of~$X$ is guaranteed if the assumption on~$X$ 
is strengthened to being weakly Lindelöf determined (Corollary~\ref{cor:WLD}).

\subsection*{Notation and terminology}

We refer to \cite{bou-J,die-uhl-J} (resp. \cite{Musial,van}) for detailed information on vector measures and the RNP (resp.
Pettis integrability and the WRNP). 
All our linear spaces are real. An {\em operator} between Banach spaces is a continuous linear map. 
By a {\em subspace} of a Banach space we mean a closed linear subspace.
No closedness is assumed when we just talk about ``linear subspaces''.
We write $B_X=\{x\in X:\|x\|\leq 1\}$ (the closed unit ball of~$X$). The topological dual of~$X$
is denoted by~$X^*$. The evaluation of $x^*\in X^*$ at $x\in X$ is denoted by either $x^*(x)$
or $\langle x^*,x\rangle$. A set $\Gamma \sub X^*$ is {\em total} (over~$X$) if it separates the points of~$X$, i.e. for every $x\in X\setminus \{0\}$ 
there is $x^*\in \Gamma$ such that $x^*(x)\neq 0$. A linear subspace $\Gamma \sub X^*$ is said to be {\em norming}
if the formula 
$$
	|||x|||=\sup\{x^*(x):\, x^*\in \Gamma \cap B_{X^*}\}, \quad x\in X,
$$
defines an equivalent norm on~$X$. The weak topology on~$X$ and the weak$^*$ topology on~$X^*$ are denoted by~$w$ and~$w^*$, respectively. 
Given another Banach space~$Z$, we write $X\not\supseteq Z$ if $X$ contains no subspace isomorphic to~$Z$. 
The convex hull and the linear span of a set $D \sub X$ are denoted by ${\rm co}(D)$
and ${\rm span}(D)$, respectively. The Banach space~$X$ is said to be {\em weakly Lindel\"{o}f determined} (WLD)
if $(B_{X^*},w^*)$ is a {\em Corson} compact, i.e. it is homeomorphic to a set~$K \subset [-1,1]^I$ (for some non-empty set~$I$),
equipped with the product topology, in such a way that $\{i\in I:k(i)\neq 0\}$ is countable for every $k\in K$.
Every weakly $\mathcal{K}$-analytic (e.g. weakly compactly generated) 
Banach space is WLD, but the converse does not hold in general, see e.g. \cite{fab-alt-JJ} for more
information on WLD spaces. Finally, we recall that a locally convex Hausdorff space~$E$ is said to have the {\em Mazur property} if every
sequentially continuous linear functional $\varphi:E \to \mathbb{R}$ is continuous.

\section{Preliminaries on scalar measurability and total sets}\label{section:measurability}

Given any set $\Gamma \sub X^*$, we denote by $\sigma(\Gamma)$ the smallest $\sigma$-algebra on~$X$ for which each $x^*\in \Gamma$ is measurable
(as a real-valued function on~$X$). 
According to a result of Edgar (see~\cite[Theorem~2.3]{edgar1}), 
$\sigma(X^*)$ coincides with the Baire $\sigma$-algebra of~$(X,w)$, that is, $\sigma(X^*)=\Ba(X,w)$
(cf. \cite[Theorem~2.1]{Musial}). 
The following Banach space property was introduced by Gulisashvili~\cite{gul-J}:

\begin{defi}\label{defi:D} 
The Banach space $X$ is said to have {\em property ($\cD$)} if the equality $\Ba(X,w)=\sigma(\Gamma)$ 
holds for every total set $\Gamma\sub X^*$.
\end{defi}

Equivalently, $X$ has property ($\cD$) if and only if for every total set $\Gamma \sub X^*$ and every measurable space $(\tilde{\Omega},\tilde{\Sigma})$ we have:
\begin{quote} 
a function $f:\tilde{\Omega}\to X$ is {\em $\Gamma$-scalarly measurable} (i.e. $\langle x^*,f\rangle$ is measurable for every $x^*\in \Gamma$)
if and only if it is {\em scalarly measurable} (i.e. $\langle x^*,f\rangle$ is measurable for every $x^*\in X^*$).
\end{quote}

Any Banach space with $w^*$-angelic dual has property~($\cD$), see \cite[Theorem~1]{gul-J}. 
The converse fails in general: this is witnessed by the Johnson-Lindenstrauss space $JL_2(\mathcal{F})$ associated
to any maximal almost disjoint family $\mathcal{F}$ of infinite subsets of~$\mathbb{N}$ (see the introduction of \cite{avi-mar-rod}
and the references therein). More generally, {\em property~($\cD'$) implies property~($\cD$)}, see \cite[Proposition~12]{pli3}.

\begin{lem}\label{lem:subspaces}
If $X$ has property ($\cD$), then any subspace of~$X$ has property ($\cD$). 
\end{lem}
\begin{proof}
Let $Y \sub X$ be a subspace. Given any set $\Gamma \sub Y^*$, for each $\gamma \in \Gamma$
we take $\tilde{\gamma}\in X^*$ such that $\tilde{\gamma}|_Y=\gamma$. Define $\tilde{\Gamma}:=\{\tilde{\gamma}:\gamma\in \Gamma\}\sub X^*$.
The following three statements hold true without the additional assumption on~$X$: 
\begin{enumerate}
\item[(i)] $\Ba(Y,w)=\{Y\cap A: A\in \Ba(X,w)\}$, as an application of the aforementioned result of Edgar
(see e.g. \cite[Corollary~2.2]{Musial}).
\item[(ii)] $\sigma(\Gamma)=\{Y\cap A: A\in \sigma(\tilde{\Gamma}\cup Y^\perp)\}$. Indeed, since every
element of $\tilde{\Gamma}\cup Y^\perp$ is $\sigma(\Gamma)$-measurable when restricted to~$Y$, the inclusion
operator from~$Y$ into~$X$ is ($\sigma(\Gamma)$-($\sigma(\tilde{\Gamma}\cup Y^\perp)$)-measurable, that is,
$$
	\mathcal{A}:=\{Y\cap A: \, A\in \sigma(\tilde{\Gamma}\cup Y^\perp)\} \sub \sigma(\Gamma).
$$
On the other hand, $\mathcal{A}$ is a $\sigma$-algebra on~$Y$ for which every element of~$\Gamma$ is $\mathcal{A}$-measurable,
hence $\sigma(\Gamma)=\mathcal{A}$. 
\item[(iii)] If $\Gamma$ is total (over~$Y$), then $\tilde{\Gamma}\cup Y^\perp$ is total (over~$X$). Indeed, just bear in mind that
$Y^{\perp\perp}\cap X=Y$ (by the Hahn-Banach separation theorem).
\end{enumerate} 
So, if $X$ has property~($\cD$) and $\Gamma$ is total, then  (i), (ii) and~(iii) 
ensure that $\Ba(Y,w)=\sigma(\Gamma)$.
This shows that $Y$ has property~($\cD$).
\end{proof}

It is easy to check that any WLD Banach space has $w^*$-angelic dual (so it has property~($\cD$)).
On the other hand, a non-trivial fact is that any WLD Banach space has a Markushevich basis (see e.g. \cite[Theorem~5.37]{fab-alt-JJ}).
The following result might be compared with that of Vanderwerff, Whitfield and Zizler 
that a Banach space having a Markushevich basis and Corson's property~($\mathcal{C}$) is WLD
(see \cite[Theorem~3.3]{van-alt}, cf. \cite[Theorem~5.37]{fab-alt-JJ}).

\begin{pro}\label{pro:Mbasis}
If $X$ has a Markushevich basis and property~($\cD$), then $X$ is WLD.
\end{pro}
\begin{proof} 
Let $\{(x_i,x_i^*):i\in I\}\sub X\times X^*$ be a Markushevich basis of~$X$. Then $\{x_i^*:i\in I\}$ is total and property ($\cD$) ensures that
$\Ba(X,w)=\sigma(\{x_i^*:i\in I\})$.
Take any $x^*\in X^*$. By \cite[Lemma~3.5]{rod-ver}, there is a countable
set $I_{x^*}\sub I$ such that $x^* \in \overline{{\rm span}(\{x_i^*: \, i\in I_{x^*}\})}^{w^*}$,
hence $x^*(x_i)=0$ for every $i\in I\setminus I_{x^*}$. This implies that $X$ is WLD
(see e.g. \cite[Theorem~5.37]{fab-alt-JJ}).
\end{proof}

Bearing in mind that $\ell_1(\omega_1)$ is not WLD, Lemma~\ref{lem:subspaces} and Proposition~\ref{pro:Mbasis}
yield the following known result (cf. \cite[Lemma~11]{pli3}).

\begin{cor}\label{cor:Plichko-Gulisashvili-l1}
A Banach space having property~($\cD$) contains no subspace isomorphic to~$\ell_1(\omega_1)$.
\end{cor}

To the best of our knowledge, the next question remains open:

\begin{question}[Gulisashvili, \cite{gul-J}]\label{question:G}
Is $(X^{**},w^*)$ angelic whenever $X^*$ has property~($\cD$)?
\end{question}

This question has affirmative answer if $X$ is separable (see \cite[Theorem~2]{gul-J}).
Indeed, the Odell-Rosenthal and Bourgain-Fremlin-Talagrand theorems ensure
that $(X^{**},w^*)$ is angelic whenever $X$ is a separable Banach space such that $X\not \supseteq \ell^1$
(see e.g. \cite[Theorem~4.1]{van}). On the other hand, property~($\cD$) of~$X^*$ implies that $X\not\supseteq \ell^1$, even 
for a non-separable~$X$. Indeed, this follows from the following result (see \cite[Proposition~3.9]{rod-ver}):

\begin{fact}\label{fact:properties-dual-D}
Let us consider the following statements:
\begin{enumerate}
\item[(i)] $X^*$ has property~($\mathcal{D}$).
\item[(ii)] $\Ba(X^*,w)=\sigma(X)$.
\item[(iii)] $X$ is $w^*$-sequentially dense in~$X^{**}$.
\item[(iv)] $X\not\supseteq \ell_1$.
\end{enumerate}
Then (i)$\impli$(ii)$\Leftrightarrow$(iii)$\impli$(iv).
\end{fact}

It is known that~$X^*$ is WLD if and only if $X$ is Asplund
and $X$ is $w^*$-sequentially dense in~$X^{**}$, see \cite[Theorem~III-4, Remarks~III-6]{dev-god} and \cite[Corollary~8]{ori1}.
Recall that $X$ is said to be {\em Asplund} if every subspace of~$X$ has separable dual or, equivalently, $X^*$ has the RNP
(see e.g. \cite[p.~198]{die-uhl-J}). Therefore, Question~\ref{question:G} has affirmative answer for Asplund spaces:

\begin{cor}\label{cor:Asplund}
If $X$ is Asplund and $X^*$ has property~($\cD$), then $X^*$ is WLD. 
\end{cor}

In particular, since a Banach lattice is Asplund if (and only if) it contains no subspace isomorphic to~$\ell_1$ 
(see \cite[p.~95]{die-uhl-J} and \cite[Theorem~7]{gho-saa-J}), we have:

\begin{cor}\label{cor:lattice}
If $X$ is a Banach lattice and $X^*$ has property~($\cD$), then $X^*$ is WLD. 
\end{cor}

\section{Integration and total sets}\label{section:integration}

Throughout this section $\Gamma\sub X^*$ is a total linear subspace.
Given a function $f:\Omega \to X$, we write
$$
	Z_{f,D}:=\{\langle x^*, f \rangle: \, x^*\in D\}
$$
for any set $D \sub X^*$. Note that if $f$ is $\Gamma$-integrable, then the map 
$$
	I_f: \Sigma \to X, 
	\quad I_f(A):=\int_A f\, d\nu,
$$ 
is a finitely additive vector measure vanishing on $\nu$-null sets. It is countably additive (and it is called the {\em indefinite Pettis integral}
of~$f$) in the particular case $\Gamma=X^*$. 

Statement~(i) of our next lemma is an application of a classical result of Diestel and Faires~\cite{die-fai},
while part~(iii) is similar to \cite[Lemma~3.1]{cas-rod-2}.
Recall first that a set $H\sub L_1(\nu)$ is called {\em uniformly integrable} if it is bounded
and for every $\epsilon>0$ there is $\delta>0$ such that $\sup_{h\in H}\int_A |h| \, d\nu\leq \epsilon$
for every $A\in \Sigma$ with $\nu(A)\leq\delta$.

\begin{lem}\label{lem:UniformlyIntegrableZf}
Let $f:\Omega \to X$ be a $\Gamma$-integrable function. 
\begin{enumerate}
\item[(i)] If $X \not \supseteq \ell_\infty$, then $I_f$ is countably additive.
\item[(ii)]  If $I_f$ is countably additive, then $Z_{f,\Gamma \cap B_{X^*}}$ is uniformly integrable.
\item[(iii)] If $I_f$ is countably additive and there is a partition $\Omega=\bigcup_{n\in \Nat} A_n$ into countably many measurable sets
such that each restriction $f|_{A_n}$ is Pettis integrable, then $f$ is Pettis integrable.
\end{enumerate}
\end{lem}
\begin{proof} (i) follows from \cite[Theorem~1.1]{die-fai} (cf. \cite[p.~23, Corollary~7]{die-uhl-J}). 

(ii) Since $I_f$ is countably additive, it is bounded
and $\lim_{\nu(A)\to 0} \|I_f\|(A)=0$, where $\|I_f\|(\cdot)$ 
denotes the semivariation of~$I_f$ (see e.g. \cite[p.~10, Theorem~1]{die-uhl-J}).
Now, the uniform integrability of $Z_{f,\Gamma \cap B_{X^*}}$ follows from the inequality
$$
	\sup_{x^*\in \Gamma \cap B_{X^*}} \int_A |\langle x^*, f \rangle| \, d\nu
	\leq \|I_f\|(A)
	\quad \mbox{for all }A\in \Sigma.
$$

(iii) We write $\Sigma_{A}:=\{B\cap A:B\in \Sigma\}$ for every $A\in \Sigma$.
Clearly, $f$ is scalarly measurable. We first prove that $f$ is {\em scalarly integrable}. Fix $x^*\in X^*$. Take any $N\in \Nat$ and
set $B_N:=\bigcup_{n=1}^N A_n$. Note that $f|_{B_N}$ is Pettis integrable and its indefinite Pettis integral
coincides with $I_f$ on $\Sigma_{B_N}$. Hence
$$
	\int_{B_N} |\langle x^*, f \rangle| \, d\nu\leq \|I_f\|(B_N)\leq \|I_f\|(\Omega) < \infty.
$$
As $N\in \Nat$ is arbitrary, $\langle x^*, f \rangle$ is integrable. This shows that $f$ is scalarly integrable.

Fix $E\in \Sigma$. Since $I_f$ is countably additive, the series $\sum_{n\in \Nat} I_f(E\cap A_n)$ is unconditionally convergent in~$X$ and for each $x^*\in X^*$ we have
$$
	x^*\Big(\sum_{n\in \Nat} I_f(E\cap A_n)\Big)=
	\sum_{n\in \Nat} x^*\big(I_f(E\cap A_n)\big)\stackrel{(*)}{=}
	\sum_{n\in \Nat} \int_{E\cap A_n} \langle x^*, f \rangle \, d\nu=
	\int_E \langle x^*, f \rangle \, d\nu,
$$
where equality~($*$) holds because $I_f|_{\Sigma_{A_n}}$ is the indefinite Pettis integral of $f|_{A_n}$ for every $n\in \Nat$.
This proves that $f$ is Pettis integrable.
\end{proof}

The Banach space $X$ is said to have the {\em $\nu$-Pettis Integral Property} ($\nu$-PIP)
if every scalarly measurable and scalarly bounded function $f:\Omega \to X$ is Pettis integrable. 
Recall that a function $f:\Omega \to X$ is called {\em scalarly bounded} if there is a constant $c>0$
such that for each $x^*\in X^*$ we have $|\langle x^*,f\rangle|\leq c\|x^*\|$ $\nu$-a.e. (the exceptional set depending on~$x^*$).
A Banach space is said to have the PIP if it has the $\nu$-PIP with respect to any complete probability measure~$\nu$. In general:
\begin{center}
$X$ has property~($\cD'$) \ $\impli$ \ $(X^*,w^*)$ has the Mazur property \ $\impli$ \ $X$ has the PIP.
\end{center}
Indeed, the first implication is easy to check, while the second one goes back to~\cite{edgar2} (cf. Theorem~\ref{theo:ScalarGamma}(i) below).

\begin{theo}\label{theo:Dprime}
Suppose $X$ has property~($\cD'$). Then every $\Gamma$-integrable
function $f:\Omega \to X$ is Pettis integrable.
\end{theo}
\begin{proof}
Since $X$ has property~($\cD$), $f$ is scalarly measurable. Then there is 
a partition $\Omega=\bigcup_{n\in \Nat} A_n$ into countably many measurable sets
such that $f|_{A_n}$ is scalarly bounded for every $n\in \Nat$ (see e.g. \cite[Proposition~3.1]{Musial}).
On the other hand, $X$~has the PIP acording to the comments preceeding the theorem. Therefore, each $f|_{A_n}$ is Pettis integrable.
Bearing in mind that a Banach space having property~($\cD$) cannot contain subspaces isomorphic to~$\ell_\infty$
(by Corollary~\ref{cor:Plichko-Gulisashvili-l1} and the fact that $\ell_1(\mathfrak{c})$ embeds
isomorphically into~$\ell_\infty$), an appeal to Lemma~\ref{lem:UniformlyIntegrableZf}
allows us to conclude that $f$ is Pettis integrable.
\end{proof}

\begin{rem}
From the proof of Theorem~\ref{theo:Dprime} it follows that the result still holds true if $X$ has property~($\cD$) and the $\nu$-PIP. 
There exist Banach spaces having the PIP but failing property~($\cD$), like $\ell_1(\omega_1)$ (see \cite[Theorem~5.10]{edgar2}, 
cf. \cite[Proposition~7.2]{Musial}).
However, we do not know an example of a Banach space having property~($\cD$) but failing the PIP.
\end{rem}

\begin{question}\label{question:PIP}
Does property~($\cD$) imply the PIP?
\end{question}

The previous question has affirmative answer for dual spaces. Indeed, if $X^*$
has property~($\cD$), then $X$ is $w^*$-sequentially dense in~$X^{**}$
(Fact~\ref{fact:properties-dual-D}) and so $(X^{**},w^*)$ has the Mazur property, as
it can be easily checked. As a consequence:

\begin{cor}\label{cor:dual}
Suppose $X^*$ has property~($\cD$) and let $\tilde{\Gamma} \sub X^{**}$ be a total linear subspace. 
Then every $\tilde{\Gamma}$-integrable function $f:\Omega \to X^*$ is Pettis integrable.
\end{cor}

Given any set $D \sub X^*$, we denote by $S_1(D) \sub X^*$ the set of all limits of $w^*$-convergent 
sequences contained in~$D$. More generally, for any ordinal $\alpha\leq \omega_1$, the $\alpha$-th $w^*$-sequential closure $S_\alpha(D)$
is defined by transfinite induction as follows: $S_0(D):=D$, $S_{\alpha}(D):=S_1(S_\beta(D))$ if $\alpha=\beta+1$ and $S_{\alpha}(D):=\bigcup_{\beta<\alpha}S_\beta(D)$
if $\alpha$ is a limit ordinal. Then $S_{\omega_1}(D)$ is the smallest $w^*$-sequentially closed subset of~$X^*$ containing~$D$.
Clearly, the Banach space $X$ has property ($\cD'$) (resp. ($\mathcal{E}'$)) 
if and only if $S_{\omega_1}(D)=\overline{D}^{w^*}$ for every subspace (resp. convex bounded set) $D \sub X^*$.

\begin{lem}\label{lem:ScalarGamma}
Let $f:\Omega \to X$ be a $\Gamma$-scalarly integrable function. For each $A\in \Sigma$, let $\varphi_{f,A}: \Gamma \to \mathbb{R}$ be the linear functional 
defined by
$$
	\varphi_{f,A}(x^*):=\int_A \langle x^*, f \rangle \, d\nu \quad\mbox{for all }x^*\in \Gamma.
$$
The following statements hold:
\begin{enumerate}
\item[(i)] $f$ is $\Gamma$-integrable if and only if $\varphi_{f,A}$ is $w^*$-continuous for every $A\in \Sigma$.
\item[(ii)] Let $D \sub \Gamma$. If $Z_{f,D}$ is uniformly integrable, then: 
\begin{itemize}
\item[(ii.1)] $\varphi_{f,A}$ is $w^*$-sequentially continuous on~$D$ for every $A\in \Sigma$;
\item[(ii.2)] $Z_{f,S_{\omega_1}(D)}$ is a uniformly integrable subset of~$L_1(\nu)$.
\end{itemize}
\end{enumerate}
\end{lem}
\begin{proof} (i) The ``only if'' part is obvious, while the ``if'' part follows from the fact that any
$w^*$-continuous linear functional $\varphi:\Gamma \to \mathbb{R}$ is induced by some $x\in X$
via the formula $\varphi(x^*)=x^*(x)$ for all $x^*\in \Gamma$ (see e.g. \cite[\S10.4]{kot}).

(ii.1) Let $(x_n^*)_{n\in \Nat}$ be a sequence in~$D$
which $w^*$-converges to some~$x^*\in D$. Then $(\langle x_n^*,f\rangle)_{n\in \Nat}$ is uniformly integrable
and $\lim_{n\to \infty}\langle x_n^*,f\rangle=\langle x^*,f\rangle$ pointwise on~$\Omega$. By 
Vitali's convergence theorem, we have $\lim_{n\to \infty}\|\langle x_n^*,f \rangle-\langle x^*,f\rangle\|_{L_1(\nu)}=0$.
Hence, $\lim_{n\to \infty}\varphi_{f,A}(x_n^*)=\varphi_{f,A}(x^*)$ for every $A\in \Sigma$.

(ii.2) Fix $c>0$ such that $Z_{f,D}\sub cB_{L_1(\nu)}$ and
a function $\delta: (0,\infty)\to (0,\infty)$ such that
$$
	\sup_{\substack{A\in \Sigma\\\nu(A)\leq \delta(\epsilon)}} \sup_{x^*\in D}
	\int_A |\langle x^*,f \rangle| \, d\nu\leq \epsilon
	\quad \mbox{for every }\epsilon>0.
$$
We will prove, by transfinite induction, that for each $\alpha\leq \omega_1$ we have
\begin{eqnarray*}
	(p_\alpha) &\quad \qquad& Z_{f,S_\alpha(D)} \sub c B_{L_1(\nu)} \qquad \mbox{ and}\\
	(q_\alpha) &\quad \qquad& \sup_{\substack{A\in \Sigma\\\nu(A)\leq \delta(\epsilon)}} \sup_{x^*\in S_\alpha(D)}
	\int_A |\langle x^*,f \rangle| \, d\nu\leq \epsilon
	\quad \mbox{for every }\epsilon>0.
\end{eqnarray*}
The case $\alpha=0$ is obvious. Suppose now that $0<\alpha\leq \omega_1$ and that $(p_\beta)$
and~$(q_\beta)$ hold true for every $\beta<\alpha$. If $\alpha$ is a limit,
then $S_\alpha(D)=\bigcup_{\beta<\alpha}S_\beta(D)$ and so $(p_\alpha)$
and~$(q_\alpha)$ also hold. Suppose on the contrary that $\alpha=\beta+1$.
Fix an arbitrary $x^*\in S_\alpha(D)=S_1(S_\beta(D))$. Then
there is a sequence $(x_n^*)_{n\in \Nat}$ in~$S_\beta(D)$ which $w^*$-converges to~$x^*$,
hence $\lim_{n\to \infty}\langle x_n^*,f \rangle=\langle x^*,f\rangle$ pointwise on~$\Omega$.
Since $Z_{f,S_\beta(D)}$ is uniformly integrable
(by $(p_\beta)$ and~$(q_\beta)$), we can apply Vitali's convergence theorem 
to conclude that $\langle x^*,f\rangle\in L_1(\nu)$ and $\lim_{n\to \infty}\|\langle x_n^*,f \rangle-\langle x^*,f\rangle\|_{L_1(\nu)}=0$.
Clearly, this shows that $(p_\alpha)$ and~$(q_\alpha)$ hold. 
\end{proof}

\begin{theo}\label{theo:ScalarGamma}
Let $f:\Omega \to X$ be a $\Gamma$-scalarly integrable function such that $Z_{f,\Gamma \cap B_{X^*}}$ 
is uniformly integrable.
\begin{enumerate}
\item[(i)] If $(\Gamma,w^*)$ has the Mazur property, then $f$ is $\Gamma$-integrable. 
\item[(ii)] If $X$ has property~($\mathcal{E}'$) and $\Gamma$ is norming, then $f$ is Pettis integrable.
\end{enumerate}
\end{theo}
\begin{proof}
(i) follows at once from Lemma~\ref{lem:ScalarGamma}.

(ii) The fact that $\Gamma$ is norming is equivalent to saying that $\overline{\Gamma \cap B_{X^*}}^{w^*}
\supseteq c B_{X^*}$ for some $c>0$ (by the Hahn-Banach separation theorem). Since $X$ has property~($\mathcal{E}'$), we have 
$S_{\omega_1}(\Gamma \cap B_{X^*}) =\overline{\Gamma \cap B_{X^*}}^{w^*}\supseteq cB_{X^*}$.
An appeal to Lemma~\ref{lem:ScalarGamma}(ii.2) ensures that $f$ is scalarly integrable and $Z_{f,B_{X^*}}$ is
uniformly integrable. Since $(X^*,w^*)$ has the Mazur property (by property~($\cD'$) of~$X$),
statement~(i) (applied to~$\Gamma=X^*$) implies that $f$ is Pettis integrable. 
\end{proof}

\begin{rem}\label{rem:CascalesRodriguez}
Theorem~\ref{theo:ScalarGamma}(ii) generalizes an earlier analogous result for spaces with $w^*$-angelic dual, 
see \cite[pp.~551--552]{cas-rod-2}.
\end{rem}

The Mackey topology $\mu(X,\Gamma)$ is defined as the (locally convex Hausdorff) topology on~$X$ of uniform convergence on absolutely convex $w^*$-compact subsets of~$\Gamma$.  
When $\Gamma$ is norm-closed, $(X,\mu(X,\Gamma))$ is complete if (and only if) it is quasi-complete (see e.g. \cite{avi-gui-mar-rod}).
This completeness assumption was used by Kunze~\cite{kun} to find conditions ensuring the $\Gamma$-integrability
of a $\Gamma$-scalarly integrable function provided that $\Gamma$ is norming and norm-closed. After Kunze's work, the completeness of $(X,\mu(X,\Gamma))$
has been discussed in \cite{avi-gui-mar-rod,bon-cas,gui-mon,gui-mon-ziz}. For instance, $(X,\mu(X,\Gamma))$ is complete whenever $(X^*,w^*)$ is angelic
and $\Gamma$ is norming and norm-closed (see \cite[Theorem~4]{gui-mon-ziz}). There is also a connection between the completeness
of $(X,\mu(X,\Gamma))$ and the Mazur property of~$(\Gamma,w^*)$, see \cite{gui-mon-ziz}. 

The following result is a refinement of \cite[Theorem~4.4]{kun}. We denote by 
$\sigma(X,\Gamma)$ the (locally convex Hausdorff) topology on~$X$ of pointwise convergence on~$\Gamma$.

\begin{theo}\label{theo:Kunze}
Suppose $(X,\mu(X,\Gamma))$ is complete. Let $f:\Omega \to X$ be a $\Gamma$-scalarly integrable function such that 
$Z_{f,\Gamma\cap B_{X^*}}$ is uniformly integrable and there is a 
$\sigma(X,\Gamma)$-separable linear subspace $X_0 \sub X$ such that $f(\omega)\in X_0$ for $\mu$-a.e. $\omega \in \Omega$. 
Then $f$ is $\Gamma$-integrable.
\end{theo}

To deal with Theorem~\ref{theo:Kunze} we need the following folk lemma:

\begin{lem}\label{lem:Separable}
If $X$ is $\sigma(X,\Gamma)$-separable, then every $w^*$-compact subset of~$\Gamma$ is $w^*$-metrizable.
\end{lem}
\begin{proof} Let $(x_n)_{n\in \Nat}$ be a $\sigma(X,\Gamma)$-dense sequence in~$X$. The map
$$
	\Gamma \ni x^* \mapsto (x^*(x_n))_{n\in \Nat} \in \mathbb{R}^\mathbb{N}
$$  
is ($w^*$-pointwise)-continuous and injective. Hence any $w^*$-compact subset of~$\Gamma$ is homeomorphic to a (compact) 
subset of the metrizable topological space $\mathbb{R}^\mathbb{N}$. 
\end{proof}

\begin{proof}[Proof of Theorem~\ref{theo:Kunze}] 
We will prove that $f$ is $\Gamma$-integrable by checking that $\varphi_{f,A}$ is $w^*$-continuous
for every $A\in \Sigma$ (see Lemma~\ref{lem:ScalarGamma}(i)).

Note that $X_1:=\overline{X_0}^{\sigma(X,\Gamma)}$ is a subspace of~$X$.
We denote by $r:X^*\to X_1^*$ the restriction operator, i.e. $r(x^*):=x^*|_{X_1}$ for all $x^*\in X^*$. Then
$$
	\Gamma_1:=r(\Gamma)=\{x^*|_{X_1}: \, x^*\in \Gamma\}\sub X_1^*
$$ 
is a total linear subspace (over~$X_1$). Of course, we can assume without loss of generality
that $f(\Omega)\sub X_1$, so $f$ can be seen as an $X_1$-valued $\Gamma_1$-scalarly integrable function.

Fix $A\in \Sigma$. Since $(X,\mu(X,\Gamma))$ is complete, in order to check that $\varphi:=\varphi_{f,A}$ is $w^*$-continuous it suffices 
to show that the restriction $\varphi|_K$ is $w^*$-continuous for each absolutely convex $w^*$-compact set~$K \sub \Gamma$
(see e.g. \cite[\S21.9]{kot}). Since $r$ is ($w^*$-$w^*$)-continuous, $r(K)$ is a $w^*$-compact
subset of~$\Gamma_1$. Note that $\varphi|_{K}$ factors as
$$
	\xymatrix{K\ar_{r|_K}[d]\ar[rr]^{\varphi|_K}&&\mathbb{R}\\
	r(K)\ar_{\psi}[urr]&& }
$$
where $\psi(y^*):=\int_A \langle y^*,f \rangle \, d\nu$ for all $y^*\in r(K)$.
Observe that $Z_{f,K}$ is uniformly integrable
(because $K$ is a bounded subset of~$\Gamma$ and $Z_{f,\Gamma \cap B_{X^*}}$
is uniformly integrable). From Lemma~\ref{lem:ScalarGamma}(ii.1) (applied to~$f$ as an $X_1$-valued $\Gamma_1$-scalarly
integrable function and $D=r(K)$) it follows that $\psi$ is $w^*$-sequentially continuous on~$r(K)$.
Since $r(K)$ is $w^*$-metrizable (bear in mind the $\sigma(X_1,\Gamma_1)$-separability of~$X_1$ and Lemma~\ref{lem:Separable}),
we conclude that $\psi$ is $w^*$-continuous, and so is~$\varphi|_K$.
The proof is finished.
\end{proof}

\section{Integration and operators}\label{section:operators}

Throughout this section $Y$ is a Banach space. 
Any operator $T:X\to Y$ is ($\Ba(X,w)$-$\Ba(Y,w)$)-measurable, in the sense that
$T^{-1}(B) \in \Ba(X,w)$ for every $B\in \Ba(Y,w)$.
In fact, we have the following lemma, whose proof
is included for the sake of completeness.

\begin{lem}\label{lem:sigma-adjoint-total}
Let $T:X \to Y$ be an operator. Then 
$$
	\sigma(T^*(Y^*))=\big\{T^{-1}(B): \, B\in \Ba(Y,w)\big\}.
$$
\end{lem}
\begin{proof}
Note that $\mathcal{A}:=\{T^{-1}(B): \, B\in \Ba(Y,w)\}$ is a $\sigma$-algebra on~$X$.
Since $T^*(y^*)=y^*\circ T: X \to \mathbb{R}$ is $\mathcal{A}$-measurable for every $y^*\in Y^*$,
the inclusion $\sigma(T^*(Y^*)) \sub \mathcal{A}$ holds. On the other hand, define
$$
	\mathcal{B}:=\big\{B\in \Ba(Y,w): \, T^{-1}(B)\in \sigma(T^*(Y^*))\big\},
$$
which is clearly a $\sigma$-algebra on~$Y$. Since each $y^*\in Y^*$ is $\mathcal{B}$-measurable,
we have $\Ba(Y,w)=\mathcal{B}$ and so $\sigma(T^*(Y^*))=\mathcal{A}$.
\end{proof}

\begin{cor}\label{cor:characterizationWBE}
Let $T:X \to Y$ be an operator. The following statements are equivalent:
\begin{enumerate}
\item[(i)] $\Ba(X,w)=\{T^{-1}(B): \, B\in \Ba(Y,w)\}$.
\item[(ii)] For every measurable space $(\tilde{\Omega},\tilde{\Sigma})$ and every function $f:\tilde{\Omega} \to X$
we have: $f$ is scalarly measurable if (and only if) $T\circ f$ is scalarly measurable. 
\end{enumerate}
\end{cor}

An injective operator satisfying the statements of Corollary~\ref{cor:characterizationWBE}
is said to be  a {\em weak Baire embedding}. This concept was considered by Andrews~\cite[p.~152]{and} in the particular case of adjoint operators.
Plainly, {\em if $T:X \to Y$ is an injective operator, then $T^*(Y^*)$ is a total linear subspace of~$X^*$.}
Therefore:

\begin{cor}\label{cor:Dtilde}
If $X$ has property~($\cD$), then every injective operator $T:X \to Y$ is a weak Baire embedding.
\end{cor}

As to strong measurability, in \cite[Proposition 4.3(ii)]{dev-rod} it was shown that if $X$ is WLD and $T:X\to Y$ is an injective operator, then a function
$f:\Omega \to X$ is strongly measurable if (and only if) $T\circ f:\Omega \to Y$ is strongly measurable.
The following proposition extends that result:

\begin{pro}\label{pro:Dfunctions}
Suppose $X$ is weakly measure-compact and has property~($\cD$).
Let $T:X \to Y$ be an injective operator and let $f:\Omega \to X$ be a function. If $T\circ f$ is strongly measurable, then so is~$f$.
\end{pro}

Recall that the Banach space $X$ is said to be {\em weakly measure-compact} if every probability measure $P$ on~$\Ba(X,w)$
is {\em $\tau$-smooth}, i.e. for any net $h_\alpha:X \to \mathbb{R}$ of bounded weakly continuous functions
which pointwise decreases to~$0$, we have $\lim_\alpha \int_X h_\alpha \, dP=0$. Every weakly Lindelöf (e.g. WLD) Banach space
is weakly measure-compact, see \cite[Section~4]{edgar1}.

\begin{proof}[Proof of Proposition~\ref{pro:Dfunctions}] 
By Corollary~\ref{cor:Dtilde}, $f$ is scalarly measurable. 
Since $X$ is weakly measure-compact, a result of Edgar (see~\cite[Proposition~5.4]{edgar1}, cf. \cite[Theorem~3.3]{Musial})
ensures that $f$ is {\em scalarly equivalent} to a strongly measurable function $f_0:\Omega \to X$.
Scalar equivalence means that for each $x^*\in X^*$ we have $\langle x^*,f\rangle=\langle x^*,f_0\rangle$ $\nu$-a.e. 
(the exceptional set depending on~$x^*$).
Note that $T\circ f$ and $T\circ f_0$ are scalarly equivalent and strongly measurable, hence $T\circ f = T\circ f_0$ $\nu$-a.e.
The injectivity of~$T$ implies that $f=f_0$ $\nu$-a.e., so that $f$ is strongly measurable. 
\end{proof}

\begin{rem}\label{rem:improvesDeville}
The class of weakly measure-compact Banach spaces having property~$(\mathcal{D})$
is strictly larger than that of WLD spaces. 
Indeed, there are compact Hausdorff topological spaces $K$ with the following properties:
\begin{enumerate}
\item[(i)] $K$ is scattered of height~$3$,
\item[(ii)] $C(K)$ is weakly Lindelöf,
\item[(iii)] $K$ is not Corson,
\end{enumerate}
see \cite{pol2} (cf. \cite{fer-kos-kub}). Condition~(i) implies that
$(B_{C(K)^*},w^*)$ is sequential (see \cite[Theorem~3.2]{gon3}), thus $C(K)$
has property~($\mathcal{E}'$) and so property~($\cD$). (ii)~implies that $C(K)$ is weakly measure-compact.
On the other hand, $C(K)$ is not WLD by~(iii).
\end{rem}

Obviously, if $T:X \to Y$ is an injective operator and the function $f:\Omega \to X$
is $T^*(Y^*)$-integrable, then $T\circ f$ is Pettis integrable. In the opposite direction, we have the following result:

\begin{pro}\label{pro:TX-integrability}
Let $T:X \to Y$ be a semi-embedding and let $f:\Omega \to X$ be a function for which there is a constant $c>0$ such that,
for each $y^*\in Y^*$, we have 
\begin{equation}\label{eqn:scalarboundedness}
	|\langle T^*(y^*),f \rangle| \leq c \|T^*(y^*)\| \quad \nu\mbox{-a.e.}
\end{equation}
If $T\circ f$ is Pettis integrable, then $f$ is $T^*(Y^*)$-integrable.
\end{pro}
\begin{proof} Since $T$ is a semi-embedding, $T(c B_X)$ is closed. We claim that 
\begin{equation}\label{eqn:inclusion}
	\frac{1}{\nu(A)}\int_A T\circ f \, d\nu \in T(c B_X)
\end{equation}
for every $A\in \Sigma$ with $\nu(A)>0$. Indeed, suppose this is not the case. Then the Hahn-Banach separation theorem ensures the
existence of some $y^*\in Y^*$ such that 
$$
	\frac{1}{\nu(A)}\int_A \langle y^*,T\circ f \rangle \, d\nu=
	y^*\Big(\frac{1}{\nu(A)}\int_A T\circ f \, d\nu\Big) >
	\sup_{x\in B_X} y^*(T(cx)) =c \|T^*(y^*)\|.
$$
But, on the other hand, we have
$$
	\frac{1}{\nu(A)}\int_A \langle y^*,T\circ f \rangle \, d\nu
	=
	\frac{1}{\nu(A)}\int_A \langle T^*(y^*), f \rangle \, d\nu
	\stackrel{\eqref{eqn:scalarboundedness}}{\leq}
	c\|T^*(y^*)\|,
$$  
which is a contradiction. This proves~\eqref{eqn:inclusion}. Therefore, $\int_A T\circ f \, d\nu \in T(X)$ for every $A\in \Sigma$, which clearly implies
that $f$ is $T^*(Y^*)$-integrable.
\end{proof}

\begin{rem}\label{rem:tauberian}
Let $T:X\to Y$ be a tauberian operator (i.e. $(T^{**})^{-1}(Y)=X$).
In this case, the argument of \cite[Proposition~8]{edg7} shows that
a scalarly measurable and scalarly bounded function $f:\Omega \to X$
is Pettis integrable whenever $T\circ f$ is Pettis integrable. We stress that an operator
$T:X \to Y$ is injective and tauberian if and only if the restriction $T|_Z$ is a semi-embedding for every subspace $Z \sub X$,
see \cite{nei-ros}.
\end{rem}

We arrive at the main result of this section:

\begin{theo}\label{theo:sequential}
Let $T:X \to Y$ be a semi-embedding. If $X$ has property~($\mathcal{D}'$) and $Y$ has the $\nu$-RNP,
then $X$ has the $\nu$-WRNP.
\end{theo}
\begin{proof}
Let $I:\Sigma\to X$ be a countably additive vector measure of $\sigma$-finite variation
such that $I(A)=0$ for every $A\in \Sigma$ with $\nu(A)=0$.  
We will prove the existence of a Pettis
integrable function $f:\Omega\to X$ such that $I(A)=\int_A f \, d\nu$ for all $A\in \Sigma$.
We can assume without loss of generality that there is a constant $c>0$
such that $\|I(A)\|\leq c \nu(A)$ for every $A\in \Sigma$ (see e.g. \cite[proof of Lemma~5.9]{van}).

Define 
$$
	\tilde{I}:\Sigma \to Y, 
	\quad
	\tilde{I}(A):=T(I(A)),
$$
so that $\tilde{I}$ is a countably additive vector measure satisfying
$\|\tilde{I}(A) \|\leq \|T\|c \nu(A)$ for every $A\in \Sigma$.
Since $Y$ has the $\nu$-RNP, there is a Bochner integrable function $g:\Omega \to Y$ such that
\begin{equation}\label{eqn:derivative}
	T(I(A))=\tilde{I}(A)=\int_A g\, d\nu
	\quad\mbox{for all }A\in \Sigma. 
\end{equation}
The Bochner integrability of~$g$ implies that
$$
	g(\omega) \in H:=\overline{\Big\{\frac{1}{\nu(A)}\int_A g \, d\nu: \, A\in \Sigma, \, \nu(A)>0\Big\}}^{\|\cdot\|}
	\quad\mbox{for }\nu\mbox{-a.e. }\omega \in \Omega, 
$$
see e.g. \cite[Lemma~4.3]{kup} (cf. \cite[Lemma~3.7]{cas-kad-rod-5}),
so we can assume without loss of generality that $g(\Omega) \sub H$. 
On the other hand, \eqref{eqn:derivative} yields
$$
	H\sub \overline{T(c B_X)}^{\|\cdot\|}=T(c B_X),
$$
where the equality holds because $T$ is a semi-embedding. Hence $g(\Omega) \sub T(c B_X)$, so that
there is a function $f:\Omega \to cB_X$ such that $T\circ f=g$. 

We claim that $f$ is Pettis integrable. Indeed, $f$ is $T^*(Y^*)$-integrable by Proposition~\ref{pro:TX-integrability}. 
Bearing in mind that $X$ has property~($\cD'$), an appeal to Theorem~\ref{theo:Dprime} ensures that $f$ is Pettis integrable. 
Since $T$ is injective and 
$$
	T\Big(\int_A f\, d\nu\Big)=\int_A T\circ f\, d\nu=\int_A g\, d\nu\stackrel{\eqref{eqn:derivative}}{=}T(I(A))
	\quad
	\mbox{for all }A\in \Sigma,
$$
we have $\int_A f\, d\nu=I(A)$ for every $A\in \Sigma$. This proves that
$X$ has the $\nu$-WRNP.
\end{proof}

\begin{rem}\label{rem:g}
In the previous proof, an alternative way to check the Pettis integrability of~$f$ is as follows. Since
$f$ is $T^*(Y^*)$-scalarly measurable and $X$ has property~($\cD$), $f$ is scalarly measurable. 
On the other hand, $f$ is bounded, hence the $\nu$-PIP of~$X$ ensures that $f$ is Pettis integrable.
\end{rem}

At this point it is convenient to recall that the RNP (resp. WRNP) is equivalent to the $\lambda$-RNP (resp. $\lambda$-WRNP), where
$\lambda$ is the Lebesgue measure on $[0,1]$, see e.g. \cite[p.~138, Corollary~8]{die-uhl-J} (resp. \cite[Theorem~11.3]{Musial}).

\begin{rem}\label{rem:WRNP-1} 
The conclusion of Theorem~\ref{theo:sequential} might fail if $Y$ is only assumed to have the WRNP.
Indeed, let $Y$ be any Banach space having the WRNP but failing the RNP
(e.g. $Y=JT^*$, where $JT$ is the James tree space). Since the RNP is a separably determined property
(see e.g. \cite[p.~81, Theorem~2]{die-uhl-J}), there is a separable subspace $X\sub Y$
without the RNP. Then $X$ has property~($\cD'$) and the inclusion operator from~$X$ into~$Y$
is a semi-embedding, but $X$ fails the WRNP (which is equivalent to the RNP for separable Banach spaces). 
\end{rem}

\begin{rem}\label{rem:WRNP-2}
Theorem~\ref{theo:sequential} cannot be improved by replacing the WRNP of~$X$ by the RNP,
even if $X$ has $w^*$-angelic dual. Indeed, let $Z$ be a separable Banach space not containing subspaces isomorphic to~$\ell_1$ such that 
$X:=Z^*$ is not separable (e.g. $Z=JT$). Then $(X^*,w^*)$ is angelic (as we already pointed out in Section~\ref{section:measurability}), 
$X$~fails the RNP and there is a semi-embedding $T:X \to \ell_2$
(because $X$ is the dual of a separable Banach space, see e.g. \cite[Lemma 4.1.12]{bou-J}).
\end{rem}

The $\nu$-WRNP and the $\nu$-RNP are equivalent for weakly measure-compact Banach spaces
(thanks to \cite[Proposition~5.4]{edgar1}). Since every WLD Banach space is weakly measure-compact
and has property~$(\cD'$), from Theorem~\ref{theo:sequential} we get:

\begin{cor}\label{cor:WLD}
Let $T:X \to Y$ be a semi-embedding. If $X$ is WLD and $Y$ has the $\nu$-RNP,
then $X$ has the $\nu$-RNP.
\end{cor}

\subsection*{Acknowledgements}
The author wishes to thank A.~Avil\'{e}s, A.J.~Guirao and G.~Mart\'{i}nez-Cervantes
for valuable discussions related to the topic of this paper.

\bibliographystyle{amsplain}

\end{document}